\documentclass[12pt]{amsart}
\usepackage{amsmath,amssymb,amsbsy,amsfonts,amsthm,latexsym,mathabx,
            amsopn,amstext,amsxtra,euscript,amscd,stmaryrd,mathrsfs,
            cite,array,mathtools,enumerate}
\usepackage{todonotes}        
\usepackage{url}
\usepackage[colorlinks,linkcolor=blue,anchorcolor=blue,citecolor=blue,backref=page]{hyperref}

\usepackage{color}

\renewcommand*{\backref}[1]{}
\renewcommand*{\backrefalt}[4]{%
    \ifcase #1 (Not cited.)%
    \or        (p.\,#2)%
    \else      (pp.\,#2)%
    \fi}
\begin{document}

\newtheorem{theorem}{Theorem}
\newtheorem{lemma}[theorem]{Lemma}
\newtheorem{claim}[theorem]{Claim}
\newtheorem{cor}[theorem]{Corollary}
\newtheorem{prop}[theorem]{Proposition}
\newtheorem{definition}{Definition}
\newtheorem{question}[theorem]{Open Question}
\newtheorem{conj}[theorem]{Conjecture}
\newtheorem{prob}{Problem}
\newtheorem{algorithm}[theorem]{Algorithm}
\newtheorem{remark}[theorem]{Remark}

\def\squareforqed{\hbox{\rlap{$\sqcap$}$\sqcup$}}
\def\qed{\ifmmode\squareforqed\else{\unskip\nobreak\hfil
\penalty50\hskip1em\nobreak\hfil\squareforqed
\parfillskip=0pt\finalhyphendemerits=0\endgraf}\fi}

\numberwithin{equation}{section}
\numberwithin{theorem}{section}

\newcommand{\todoI}[2][]{\todo[#1,color=yellow]{I: #2}}
\newcommand{\todoM}[2][]{\todo[#1,color=blue!60]{M: #2}}

\def\cA{{\mathcal A}}
\def\cB{{\mathcal B}}
\def\cC{{\mathcal C}}
\def\cD{{\mathcal D}}
\def\cE{{\mathcal E}}
\def\cF{{\mathcal F}}
\def\cG{{\mathcal G}}
\def\cH{{\mathcal H}}
\def\cI{{\mathcal I}}
\def\cJ{{\mathcal J}}
\def\cK{{\mathcal K}}
\def\cL{{\mathcal L}}
\def\cM{{\mathcal M}}
\def\cN{{\mathcal N}}
\def\cO{{\mathcal O}}
\def\cP{{\mathcal P}}
\def\cQ{{\mathcal Q}}
\def\cR{{\mathcal R}}
\def\cS{{\mathcal S}}
\def\cT{{\mathcal T}}
\def\cU{{\mathcal U}}
\def\cV{{\mathcal V}}
\def\cW{{\mathcal W}}
\def\cX{{\mathcal X}}
\def\cY{{\mathcal Y}}
\def\cZ{{\mathcal Z}}

\def\fA{{\mathfrak A}}
\def\fJ{{\mathfrak J}}

\def\sssum{\mathop{\sum\!\sum\!\sum}}
\def\ssum{\mathop{\sum\ldots \sum}}

\def\Xm{\cX_m}

\def\MT{\mathrm{M}}
\def\ET{\mathrm{E}}
\def\rS{\mathrm{S}}

\def \C {{\mathbb C}}
\def \F {{\mathbb F}}
\def \L {{\mathbb L}}
\def \K {{\mathbb K}}
\def \N {{\mathbb N}}
\def \R {{\mathbb R}}
\def \Q {{\mathbb Q}}
\def \Z {{\mathbb Z}}

\def\barG{\overline{\cG}}
\def\\{\cr}
\def\({\left(}
\def\){\right)}
\def\fl#1{\left\lfloor#1\right\rfloor}
\def\rf#1{\left\lceil#1\right\rceil}

\newcommand{\pfrac}[2]{{\left(\frac{#1}{#2}\right)}}

\def\rem{\mathrm{\, rem~}}

\def \Prob{{\mathrm {}}}
\def\e{\mathbf{e}}
\def\ep{{\mathbf{\,e}}_p}
\def\epp{{\mathbf{\,e}}_{p^2}}
\def\eq{{\mathbf{\,e}}_q}
\def\eM{{\mathbf{\,e}}_M}
\def\eps{\varepsilon}
\def\Res{\mathrm{Res}}
\def\vec#1{\mathbf{#1}}

\def \li {\mathrm {li}\,}

\def\ip{\overline p}
\def\ipd{\ip_d}
\def\iq{\overline q}

\def\e{{\mathbf{\,e}}}
\def\ep{{\mathbf{\,e}}_p}
\def\em{{\mathbf{\,e}}_m}

\def\mand{\qquad\mbox{and}\qquad}

\newcommand{\commM}[1]{\marginpar{%
\begin{color}{red}
\vskip-\baselineskip 
\raggedright\footnotesize
\itshape\hrule \smallskip M: #1\par\smallskip\hrule\end{color}}}

\newcommand{\commI}[1]{\marginpar{%
\begin{color}{blue}
\vskip-\baselineskip 
\raggedright\footnotesize
\itshape\hrule \smallskip I: #1\par\smallskip\hrule\end{color}}}

\title[Smooth square-free numbers  in 
progressions]{On smooth square-free numbers in arithmetic 
progressions}
\date{\today}

\author{Marc Munsch}
\address{5010 Institut f\"{u}r Analysis und Zahlentheorie
8010 Graz, Steyrergasse 30, Graz, Austria}
\email{munsch@math.tugraz.at}

\author{Igor E.  Shparlinski}

\address{Department of Pure Mathematics, University of New South Wales\\
2052 NSW, Australia.}

\email{igor.shparlinski@unsw.edu.au}


\begin{abstract} A.~Booker and C.~Pomerance (2017) have shown that any residue class modulo a 
prime $p\ge 11$ can be represented by a positive $p$-smooth square-free integer $s = p^{O(\log p)}$ with all prime factors up to $p$ and conjectured that in fact one can find such $s$ with $s = p^{O(1)}$. 
Using bounds on double Kloosterman sums due to M.~Z.~Garaev (2010) we prove this conjecture in a stronger form $s \le p^{3/2 + o(1)}$ and also consider more general versions of this question replacing $p$-smoothness of $s$ by the stronger condition of $p^{\alpha}$-smoothness. 
Using bounds on multiplicative character sums and a sieve method, we also  show that we can represent all residue classes by a positive square-free integer $s\le p^{2+o(1)}$ which is $p^{1/(4e^{1/2})+o(1)}$-smooth. 
Additionally, we obtain stronger results for almost all primes $p$. 
\end{abstract}

\maketitle

\newpage

\tableofcontents

\section{Introduction and main results}

\subsection{Motivation}

We recall that an integer $n$ is called {\it $y$-smooth\/} if all prime divisors 
of $n$ do not exceed $y$, and is called {\it square-free\/} if it is not divisible by 
a square of a prime.

Following Booker and Pomerance~\cite{BoPom}, for a prime $p $ we denote 
by $M(p)$ the smallest integer $M$ such that any residue class modulo $p$ contains a $p$-smooth
square-free representative, and set, formally, $M(p) =\infty$ if no such representative
exists. We note that by~\cite[Theorem~1]{BoPom} we have $M(p)< \infty$ for every $p\ge 11$.

It is  noted in~\cite[Section~6]{BoPom} that the argument of the proof of~\cite[Theorem~1]{BoPom} 
actually gives $M(p) =  p^{O(\log p)}$ and conjectured   that  $M(p) =  p^{O(1)}$.
Here we settle this conjecture in a stronger and more general form.

More precisely, we address the question of Booker and Pomerance~\cite{BoPom}
about the smallest value of $\alpha$ for which $M_\alpha(p) < \infty$ for a sufficiently large $p$
and show that this is true for $\alpha >  1/(4 e^{1/2})$,  by using the methods of~\cite{Harm,HaSh}.
In fact, we obtain an explicit bound on  $M_\alpha(p)$ and also extend this for composite moduli.

We also note that the area of representations of residue classes by numbers of 
prescribed arithmetic structure takes its origins in the works of 
Erd{\H o}s, Odlyzko and S{\'a}rk{\"o}zy~\cite{EOS} and Harman~\cite{Harm}, 
see also~\cite{FrKuSh, Gar2, HaSh, RaWa, Shp1, Shp2, Wal}  for more recent 
developments. 

\subsection{Main results}

For a real positive $\alpha$, we denote 
by $M_\alpha^*(q)$ the smallest integer $M$ such that any {\it reduced\/} residue class modulo 
$q$ contains  a $q^\alpha$-smooth
square-free positive representative $k\le M$, and set, formally, $M_\alpha^*(q) =\infty$ if no such representative
exists.

Our main result is

\begin{theorem}\label{thm:Malpha} For  cube-free integer $q\to \infty$ and any fixed 
$$
\alpha > \begin{cases}
1/(4 e^{1/2}), & \text{if $q$ is  cube-free,}\\
1/(3 e^{1/2}), & \text{otherwise,}
\end{cases}
$$
 we have
$$
M_{\alpha}^*(q) \le  q^{2 + o(1)}.
$$
\end{theorem}

\begin{remark} In particular, Theorem~\ref{thm:Malpha} improves on the result of Harman~\cite[Theorem~3]{Harm} which, 
for any fixed $\varepsilon > 0$, gives the existence  (without the additional 
square-freeness condition)
of a $q^{1/(4 e^{1/2})+\varepsilon}$-smooth integer $n\leq q^{9/4+\varepsilon}$ in arithmetic progressions modulo a cube-free $q$ and  a  $q^{1/(3 e^{1/2})+\varepsilon}$-smooth integer $n\leq q^{7/3+\varepsilon}$ in arithmetic progressions modulo an arbitrary $q$. 
\end{remark} 

To exhibit the main ideas behind of our approach in the simplest form and also because this 
corresponds to the original question of Booker and Pomerance~\cite{BoPom}, 
in  Theorems~\ref{thm:Malpha-p-AA}, ~\ref{thm:M1} and~\ref{thm:M1-AA} we treat only prime moduli $p$. However, all necessary ingredients are readily available 
in the case of composite moduli $q$ as well, see, for example ~\cite[Lemma~4]{MuSh},~\cite[Corollary~2.5]{FoSh}  and~\cite[Lemma~3.1]{Irv}.

For almost all primes $p$, we obtain a stronger result as Theorem~\ref{thm:Malpha} using bounds on character sums from ~\cite{MuSh}.

\begin{theorem}\label{thm:Malpha-p-AA}
As $Q\to \infty$, for any fixed $\alpha > 0$ for all  but $Q^{o(1)}$ 
primes $p\in [Q,2Q]$,  we have
$$
M_\alpha^*(p) \le  p^{2   + o(1)}.
$$
\end{theorem}

In the particular case of $p$-smoothness, we can actually do more and break the $p^2$-barrier. Using
bounds of double Kloosterman sums with prime arguments due to Garaev~\cite{Gar1},  we prove:

\begin{theorem}
\label{thm:M1}
As $p\to \infty$ we have
$$
M(p) \le p^{3/2+o(1)}. 
$$
\end{theorem}

\begin{remark} It should be noted that the main result of Balog and  Pomerance~\cite{BaPom} gives the existence (without the additional condition of 
square-freeness) of a $p$-smooth integer $n\leq p^{7/4+\varepsilon}$ in arithmetic progressions modulo $p$. 
Our method can be used to derive further extensions of the results of~\cite{BaPom}. 
\end{remark}

One of the implications of our results is that in~\cite[Corollary~7]{BoPom} 
the value of $d$ can be taken to be reasonably small. 
Using a result of Irving~\cite{Irv}, we also show that for almost all $p$ one can break 
the $3/2$-threshold of Theorem~\ref{thm:M1}, see also our comments in 
Section~\ref{sec:tight} below. 

\begin{theorem}\label{thm:M1-AA}
As $Q\to \infty$,  for all  but $o(Q/\log Q)$ 
primes $p\in [Q,2Q]$,  we have
$$
M(p) \le   p^{4/3  + o(1)}.
$$
\end{theorem}

\subsection{Some methods behind our results}

The proof of Theorem~\ref{thm:Malpha} is based on the ideas of~\cite{HaSh}, which 
are modified to accomodate the square-freeness condition and which, after some preparations in
Section~\ref{sec:sum Mob},  we develop in 
Section~\ref{sec:sumprod-AP}.  Furthermore, to make it work, instead of the 
 Burgess bound (see~\cite[Theorem~12.6]{IwKow}) used in~\cite{HaSh},  we apply some bounds from~\cite{Mun,MuTr}, 
presented in Section~\ref{sec:charsum}.  

For Theorem~\ref{thm:Malpha-p-AA} we use 
a different and more direct approach which is enabled by the fact that for almost all 
primes we have bounds of very short character sums from~\cite{MuSh}, which in turn is 
based on  some ideas of Garaev~\cite{Gar0} and which
we also present in Section~\ref{sec:charsum}.

For Theorems~\ref{thm:M1} and~\ref{thm:M1-AA} we use yet another approach
which is based on bounds of some double weighted Kloosterman-like sums from~\cite{Gar1} and~\cite{Irv}, 
respectively, see Section~\ref{sec:Ksum}. These bounds are used in 
Section~\ref{sec:congprod} to study some congruences with products of primes, 
which underlie our approach.
Furthermore, in the proof of Theorem~\ref{thm:M1-AA} we also use bounds 
for the number of small solutions of some quadratic congruences, see 
Section~\ref{sec:cong recip sq}.

We introduce some general notation in  Section~\ref{sec:note} , which we then  follow throughout the paper and collect several useful facts on arithmetic functions in  Section~\ref{sec:arith funct}.

\subsection{On the tightness of our results}
\label{sec:tight} 

Clearly, the lower bounds on $\alpha$ in Theorem~\ref{thm:Malpha} cannot be improved
until the classical  bound of Burgess~\cite{Burg} on the smallest quadratic nonresidue is improved. 

Furthermore,  the upper bound of Theorem~\ref{thm:M1} also seems to be the best possible 
one can achieve
nowadays. In fact, even without any arithmetic restrictions on positive integers $u \le U$ and $v \le V$
one can guarantee the existence of a solution to $uv \equiv a \pmod p$ only for $UV \ge p^{3/2+\varepsilon}$ 
for some fixed $\varepsilon>0$, see~\cite[Section~3.1]{Shp0} for a survey of relevant results. 

\section{Preparations}

\subsection{General notation}
 \label{sec:note}

We recall that the notations $U = O(V)$,   $U \ll V$ and $U \ll V$  are all
equivalent to the assertion that the inequality $|U|\le cV$ holds for some
constant $c>0$. Throughout the paper, the implied constants in these symbols
may occasionally, where obvious, depend on the integer parameter $r\ge 1$ and are 
absolute otherwise.

 Throughout the paper, the letter $\ell$ and $p$, with and without subscripts,  always denote primes numbers. 
 
 As usual, we use $\mu(k)$, $\tau(k)$ and $\varphi(k)$ to denote the 
 M{\" o}bius, divisor and Euler functions of an integer $k \ge 1$, respectively.

We set 
$$
\psi = 2^{1/15} \mand \xi = \psi-1
$$
and write $a \sim A$ to indicate $a \in [A,\psi A]$. 

 We also write 
 \begin{equation}
\label{eq: rho}
\rho = e^{-1/2}.
 \end{equation} 
 
 \subsection{Some properties of arithmetic functions}
 \label{sec:arith funct}
 
We recall the well-known elementary identity
 \begin{equation}\label{inversion}
\sum_{d \mid \gcd(n,q)} \mu(d) = \begin{cases} 1 &\text{if $\gcd(n,q)=1$,}\\ 0 &\text{otherwise.}
\end{cases}.
\end{equation}

We also note that by  the Mertens formula (see~\cite[Equation~(2.15)]{IwKow}), for any real $Y > X \ge 2$ we have
 \begin{equation}
\label{eq:Mer}
\sum_{X \le \ell \le Y}\frac{1}{\ell} = \log \frac{\log Y}{\log X}  +O\(\frac{1}{\log X}\).
 \end{equation} 
 
 In particular, it easily follows from~\eqref{eq:Mer} that 
  \begin{equation}
\label{eq:phi}
\frac{\varphi(k)}{k} = \prod_{ \ell \mid k}\(1-\frac{1}{\ell}\) \gg \frac{1}{\log \log k} 
 \end{equation} 
 for any integer $k \ge 3$
 
We also need the classical bound
  \begin{equation}
\label{eq:tau}
 \tau(k) = k^{o(1)}, 
 \end{equation} 
on the divisor function, see, for example,~\cite[Equation~(1.81)]{IwKow}.

We recall that by~\cite[Lemma~2.5 (2)]{Mun} we have:

\begin{lemma}\label{lem: SqFAP}
For any $M > 0$ and $q \ge 2$ we have
$$
\sum_{\substack{m\sim M \\ \gcd(m,q)=1}} \mu^2(m) = \frac{\xi}{\zeta(2)}\prod_{p\mid q}\(1+\frac{1}{p}\)^{-1} M  + O(M^{1/2}\tau(q))\,.
$$
\end{lemma}

Furthermore, by~\cite[Lemma~7]{HaSh} we have the following upper bound:

\begin{lemma}\label{lem: AP bound}
For any $M >  \log q \ge 2$ we have
$$
\sum_{\substack{m\sim M \\ \gcd(m,q)=1}} 1 \ll \frac{\varphi(q)}{q}M\,.
$$
\end{lemma}

\section{Bounds of exponential and character sum and the number of solutions to some congruences}

\subsection{Character sums}
\label{sec:charsum}

Let $\cX_q$ be the set of
{\it multiplicative\/} characters of the residue ring modulo $q \ge 1$ 
and let  $\cX_q^*=\cX_q\setminus\{\chi_0\}$ be the set of   
{\it nonprincipal\/} characters; we refer the reader to~\cite[Chapter~3]{IwKow} for the relevant background.
In particular, we make use of the following orthogonality property of 
characters, see~\cite[Section~3.2]{IwKow},
\begin{equation}
\label{eq:orth}
\frac{1}{\varphi(q)}\sum_{\chi\in \cX_q} \chi(a)
 = \begin{cases} 
 1,  & \text{if}\ a \equiv 1 \pmod q,\\
0,  & \text{otherwise,}
\end{cases}
\end{equation} 
which holds for any integer $a$ with $\gcd(a,q)=1$.

Our argument rests on the existence of a good bound for the sums
\begin{equation}
\label{eq:sum s-f}
S^\sharp_\chi(t) = \sum_{\substack{1\le s \le t\\s~\text{square-free}}} \chi(s)
\end{equation} 
of characters $\chi\in\cX_q^*$ over square-free integers $s \in [1,t]$.

In particular, we need the following bound of Munsch and Trudgian~\cite[Lemma~7]{MuTr} which has been as previously stated for $r=2$ in~\cite[Lemma~3.2]{Mun}. In fact we also formulate this in a more general form 
which cover arbitrary moduli $q$, rather than only cube-free $q$. 

\begin{lemma}
\label{lem:CharSquarfree} 
 For any integer $q$ and a positive integer $t \le q$ and 
 \begin{itemize}
\item  for any fixed integer $r\geq 2$ if $q$ cube-free,
\item for $r = 2, 3$ for any $q$,
\end{itemize}
 we have
$$
 \max_{\chi\in\cX_q^*}
\left|S^\sharp_\chi(t)\right|\le t^{1-1/r}q^{(r+1)/(4r^2)+o(1)}, 
$$
as $q\to \infty$. 
\end{lemma}

In particular, we have
\begin{cor}
\label{cor:CharSquarfree-eps} 
There exists an absolute constant $c_0>0$ such that for any real $\varepsilon>0$  
and a positive integer $t$ with
 \begin{itemize}
\item  $t \in [q^{1/4 + \varepsilon}, q]$ if $q$ cube-free,
\item  $t \in [q^{1/3 + \varepsilon}, q]$ for any $q$,
\end{itemize}
we have
$$
 \max_{\chi\in\cX_q^*}
\left|S^\sharp_\chi(t)\right|\le t^{1 - c_0\varepsilon^2 }. 
$$
\end{cor}

Finally, we need the following simple bound  which follows from 
the orthogonality of characters
and  which we  refer to as the {\it mean-value estimate for character sums\/}.

\begin{lemma} 
\label{lem:Aver}
For $N \ge 1$ and any sequence of complex numbers $a_n$ we have
$$
\sum_{\chi \in \cX} \left| \sum_{n \le N} a_n \chi(n) \right|^2 \le \varphi(q) (N/q + 1) \sum_{n \le N} |a_n|^2.
$$
\end{lemma}

We present a special case of~\cite[Lemma~4]{MuSh} in the following form 
convenient for our applications. 

\begin{lemma}
\label{lem:CharSquarfree-AA} Let $t$ and $Q$  be sufficient large positive
integers with $Q \ge t^\varepsilon$ for some fixed $\varepsilon > 0$.
Then for any $\delta< 1/4$ and
$$
 \vartheta  = \min\{ (1-2 \delta) \gamma, 2 \delta\(1- \gamma\)\}
$$
where $\gamma$ is the following fractional part
$$
\gamma = \left\{\frac{2 \log Q}{\log t}\right\}
$$
  for all but at most 
$Q^{4\delta} t^{\vartheta +o(1)}$  primes $p \le Q$ we have
$$
\max_{\chi\in\cX_p^*}
\left|S^\sharp_\chi(t)\right|\le t^{1-\delta}. 
$$
\end{lemma}

\subsection{Double Kloosterman sums with prime arguments}
\label{sec:Ksum}

 For a prime $p$,  we define $\ep(z) = \exp(2 \pi i z/p)$ and consider 
  the exponential sums
$$
W_p(a;L)= \sum_{\ell_1,\ell_2 \in \cL} \ep\(a \overline \ell_1  \overline \ell_2\),
$$
where $\cL$ is the set of primes $\ell \in [L,2L]$ with $\gcd(\ell,p) =1$ 
and for an integer $k$ with $\gcd(k,p)=1$ 
we use $\overline k$ to denote the multiplicative inverse of $k$ modulo $p$, 
that is, the unique integer with 
$$
k \overline k \equiv 1 \pmod p \mand 1 \le \overline k < p.
$$
We now record the following bound which follows from the proof of~\cite[Lemma~2.4]{Gar1}.

\begin{lemma}
\label{lem:BilinSums} For $1\le L \le p^{1/3} $ we have
$$
\left| W_p(a;L) \right| \le  L^{3/2} p^{1/8+o(1)},
$$
as $p\to \infty$. 
\end{lemma}

\begin{proof} As in~\cite[Lemma~2.4]{Gar1}, we consider a more general sum
$$
W = \sum_{k=1}^K\left|\sum_{n=1}^{N_k} \gamma_n \ep\(a  \overline  k  \overline n\)\right|, 
$$
where $\gamma_n$ are some complex numbers with $\gamma_n = p^{o(1)}$,  $n =1, \ldots, N$
and $N_k$ are some positive integers with $N_k\le N$, $k=1, \ldots K$.
Following the proof of~\cite[Lemma~2.4]{Gar1}, and using~\cite[Lemma~2.3]{Gar1} 
in full generality, we arrive to the inequality 
$$
W^8 \ll p^{1+o(1)} (KN)^4 \(K^{7/2} p^{-1/2}+K^{2}\)\(N^{7/2} p^{-1/2}+N^{2}\). 
$$
We note that for $K,N\ge p^{1/3}$ we obtain the bound $W \ll (KN)^{15/16} p^{o(1)}$ 
in~\cite[Lemma~2.4]{Gar1}, while for $ K,N \le p^{1/3}$ we arrive to 
$$
W \ll (KN)^{3/4} p^{1/8+ o(1)}.
$$
The result now follows.
\end{proof}

For almost all moduli, improving some previous results from~\cite{FoSh}, 
Irving~\cite{Irv} has shown that on average over $p$ one can improve  Lemma~\ref{lem:BilinSums}.
We present the result of~\cite[Lemma~3.1]{Irv}  
in a very special case with the averaging only over prime numbers with both variables in the same range
$[L,2L]$. 

\begin{lemma}
\label{lem:BilinSums-AA}
As $Q\to \infty$, for any fixed integer $k\ge 1$, for $1 \le L \le Q$ we have
$$
\sum_{p\in [Q,2Q]} \max_{\gcd(a,p)=1} \left| W_p(a;L) \right| 
\le Q^{1+o(1)}\(  L^{(3k-1)/(2k)} Q^{1/(2k)}+  L^{(4k-1)/(2k)}\).
$$
\end{lemma}

Hence, from Lemma~\ref{lem:BilinSums-AA}, we have:

\begin{cor}
\label{cor:BilinSums-AA} 
As $Q\to \infty$, for any fixed integer $k\ge 1$, for $1 \le L \le Q$, 
for all  but $o(Q/\log Q)$  primes $p\in [Q,2Q]$,  we have  
$$
\max_{\gcd(a,p)=1} \left| W_p(a;L) \right| \le 
\( L^{(3k-1)/(2k)}  Q^{1/(2k)}+  L^{(4k-1)/(2k)}\)  Q^{o(1)}. 
$$
\end{cor}

\subsection{Congruences with reciprocals of squares}
\label{sec:cong recip sq}

 Given an integer $r \ge 1$, a real  $U \ge 1$, and $\lambda \in \F_p$,  let $I_{r,p}(U;\lambda )$  be the number of solutions to
 the congruence
 \begin{align*}
\frac{1}{u_1^2}+ \ldots+  \frac{1}{u_r^2}  &\equiv  \frac{1}{u_{r+1}^2}+ \ldots+  \frac{1}{u_{2r}^2}+ \lambda \pmod p,\\
U \le u_1, & \ldots, u_{2r}  \le 2U, \qquad i =1,\ldots, 2r.
\end{align*}

First we observe that the standard expression of  $I_{r,p}(U;\lambda )$  via additive characters immediately implies the well-known inequality
$$
I_{r,p}(U;\lambda) \le I_{r,p}(U;  0).
$$
Hence we denote 
$$
I_{r,p}(U) =  I_{r,p}(U;  0)
$$
and concentrate on this quantity.

Heath-Brown~\cite[Lemma~1]{H-B} has given  a nontrivial  bound on $I_{r,p}(U)$, 
see also~\cite[Proposition~1]{BouGar}, however these results seems to be not strong enough 
for our  purpose. However on average over $p$ a much stronger bound is given   
by ~\cite[Lemma~3.4]{LSZ1}
(We recall that all implied constants are allowed to depend on $r$): 

\begin{lemma}
\label{lem:Irp Aver}   For any  fixed positive integer $r$ and sufficiently large real $1 \le U \le Q$,
we have
$$
\frac{1}{Q}\sum_{Q \le p \le 2Q} I_{r,p}(U) \le \(U^{2r} Q^{-1} + U^r\)Q^{o(1)}.
$$
\end{lemma}

Given two positive real numbers $U$ and $V$, we denote by 
$T_{a,p}(U,V)$ the number of solutions to the congruence
$$
u^2v \equiv a \pmod p, \qquad  1 \le u \le U, \ 1 \le v \le V.
$$



\begin{lemma}
\label{lem:TIaa}
As $Q\to \infty$, for all  but $o(Q/\log Q)$  primes $p\in [Q,2Q]$,  for any integer $a$ with 
$\gcd(a,p)=1$ and reals $U$ and $V$ with  $1\le U,V \le Q$, we have 
$$
T_{a,p}(U,V) \ll V^{1/4}(Up^{-1/4}+U^{1/2})Q^{o(1)}.
$$
\end{lemma}

\begin{proof}
In order to lighten the notation, $T$ will denote the number of solutions $T_{a,p}(U,V)$. Expanding, we get 

$$T^2 \le \#\left\{1\le u_1, u_2 \le U, 1\le b\le 2V \text{ such that } \left(\frac{1}{u_1^2}+ \frac{1}{u_2^2}\right) \equiv va^{-1} \bmod p\right\}.$$ By Cauchy-Schwarz inequality, we deduce 

\begin{align*}T^4 & \ll V \#\left\{1\le u_1,u_2,u_3,u_4 \le U \text{ such that } \frac{1}{u_1^2}+ \frac{1}{u_2^2}  \equiv  \frac{1}{u_{3}^2}+ \frac{1}{u_{4}^2} \bmod p\right\} \\ & \ll V I_{2,p}(U).  \end{align*}

Using Lemma \ref{lem:Irp Aver} with $r=2$, we deduce that for almost all primes $p\in [Q,2Q]$, we have $$I_{2,p}(U) \ll p^{o(1)} (U^{4}p^{-1} + U^2).$$ Thus, for almost all primes $p$, we get 

$$T \ll V^{1/4}(Up^{-1/4}+U^{1/2})p^{o(1)}.$$
\end{proof}


\subsection{Congruences with products of primes}
\label{sec:congprod}

Given two positive real numbers $L$ and $h$, we denote by 
$N_{a,p}(L,h)$ the number of solutions to the congruence
\begin{equation}
\label{eq:cong llu}
\ell_1 \ell_2 u \equiv a \pmod p, \qquad \ell_1,\ell_2 \in \cL, \ 1 \le u \le h,
\end{equation}
where $\cL$ is the set of primes $\ell \in [L,2L]$ with $\gcd(\ell,p) =1$. 
First we note that using  standard   techniques, we easily derive  the following asymptotic formula

\begin{lemma}
\label{lem:congr-asymp} For any integer $a$ and prime $p$ with $\gcd(a,p)=1$ and  real $h$ and  $L$ 
with $1 \le L \le p^{1/3}$ and $1\le h\le p$, 
we have 
$$
N_{a,p}(L,h) = \frac{K^2 h}{p} + O\(L^{3/2} p^{1/8+o(1)}\), 
$$
where $K = \# \cL$ is the cardinality of $\cL$.
\end{lemma}

\begin{proof}
We interpret  the  congruence~\eqref{eq:cong llu} 
as the uniformity of distribution question about the number of residues 
$a \ell_1^{-1} \ell_2^{-1}  \pmod p$ (where the inversions are modulo $p$), which 
fall in the interval $[1,h]$.

The result follows from Lemma~\ref{lem:BilinSums} applied to the sets $\cU=\cV$ which are
the sets of reciprocals modulo $p$ of $\ell \in \cL$, combined with the    the Erd{\H o}s--Tur{\'a}n
inequality, see~\cite{DrTi,KuNi}.
\end{proof} 

Similarly, from Corollary~\ref{cor:BilinSums-AA}, we derive

\begin{lemma}
\label{lem:congr-asymp-AA} 
As $Q\to \infty$, for any fixed integer $k\ge 1$, for $1 \le L \le Q$ 
for all  but $o(Q/\log Q)$  primes $p\in [Q,2Q]$,  for any integer $a$ with $\gcd(a,p)=1$ and  real $h$ 
with  $1\le h\le p$, we have 
$$
N_{a,p}(L,h) = \frac{K^2 h}{p} + O\(\( L^{(3k-1)/(2k)} p^{1/(2k)}+  L^{(4k-1)/(2k)}\) p^{o(1)}\), 
$$
where $K = \# \cL$ is the cardinality of $\cL$.
\end{lemma}

We also need the following upper bound on $N_{a,p}(L,h) $ which is better for small
values of $h$ when Lemma~\ref{lem:congr-asymp}  fails to produce any nontrivial result.

\begin{lemma}
\label{lem:congr-bound} For any integer $a$ and prime $p$ with $\gcd(a,p)=1$ and reals $1 \le L, h \le p$ 
we have 
$$
N_{a,p}(L,h) \le \(L^2 h/p +1\)p^{o(1)}.
$$
\end{lemma}

\begin{proof}  Clearly, we can assume that $1 \le a \le p$. 
Then  the  congruence~\eqref{eq:cong llu} implies that 
$\ell_1 \ell_2 u  = a  + kp$ for some non-negative $k \le 4L^2h/p$. 
Thus $k$ takes at most $4L^2 h/p +1$ possible values and 
for each of them $\ell_1$ and $\ell_2$ cant take at most $O(\log p)$ 
possible values among the divisors of $a  + kp \ge 1$.
\end{proof}

We now estimate the  average value of  $N_{a,p}(L,h)$ over a 
special set of $a$ using Lemma~\ref{lem:TIaa}.

\begin{lemma}
\label{lem:congr-bound-aver} As $Q\to \infty$,  for all  but $o(Q/\log Q)$ 
primes $p\in [Q,2Q]$,  for any integer $a$, and real  $1 \le F, L, h \le p$ 
with $F, L^2h < p$, 
for the sum
$$
R_{a,p}(F,L,h) = \sum_{F \le d \le 2F} N_{ad^{-2},p}(L,h) 
$$
we have 
$$
R_{a,p}(F,L,h) \le    \max\{F(L^2h)^{1/4}p^{-1/4}, F^{1/2}(L^2h)^{1/4}\} p^{o(1)}.
$$
\end{lemma}

\begin{proof}  We observe that the   sum $R_{a,p}(F,L,h)$
counts the number of solutions to the congruence
$$
\ell_1 \ell_2 u d^2 \equiv a \pmod p, \qquad F \le d \le 2F,\ \ell_1,\ell_2 \in \cL, \ 1 \le u \le h.
$$
Denoting $v = \ell_1 \ell_2 u$ we see from~\eqref{eq:tau}  that each such $v \in [1, L^2h]$ can be represented 
like this in at most $p^{o(1)}$ ways. Hence 
$$
R_{a,p}(F,L,h) \le   T_{a,p}(2F,L^2h)   p^{o(1)}.
$$
Thus by  Lemma~\ref{lem:TIaa} 
$$
R_{a,p}(F,L,h) \le    \max\{F(L^2h)^{1/4}p^{-1/4}, F^{1/2}(L^2h)^{1/4}\} p^{o(1)}.
$$
This concludes the proof.
\end{proof} 

We also need to bound the number of solutions of a modified version of~\eqref{eq:cong llu}.
Namely, we use  $Q_{a,p}(L,h)$ to denote the number of solutions to the congruence
\begin{equation}
\label{eq:cong ll2v}
\ell_1 \ell_2^2 v \equiv a \pmod p, \qquad \ell_1,\ell_2 \in \cL, \ 1 \le v \le h,
\end{equation}

\begin{lemma}
\label{lem:congr-boundsquare} For any integer $a$ and prime $p$ with $\gcd(a,p)=1$ and reals $ 1\le L, h \le  p$ 
with $2L h \le p$
we have 
$$
Q_{a,p}(L,h) \le \(L h/p +1\)Lp^{o(1)}.
$$
\end{lemma}

\begin{proof}  
The  congruence~\eqref{eq:cong ll2v} implies that 
$\ell_1 v  \equiv a \ell_2^{-2} \pmod p$. Since  $2L h \le p$,
for each choice $\ell_2$,  the value $\ell_1 v$ can take at most 
$Lh/p+1$ values and the result follows.
\end{proof}

\section{Multilinear sums over products in arithmetic progressions}

\subsection{Some sums with  the   M{\" o}bius function}
\label{sec:sum Mob}
 
We are now able to establish a full analogue of~\cite[Lemma~8]{HaSh}, where the summation is
only over $m$ and $n$ with square-free products $mn$.

 \begin{lemma}\label{smooth}
For integers $N \ge q^{1/4} > 1$,  real $0 < \zeta < 1$ and a positive integer $d = o\((\log N)^2/\log \log N\)$ coprime with $q$,  we have
\begin{align*} 
& \sum_{\substack{N^{\zeta} \le p \le N \\ \gcd(p,dq)=1}} \sum_{\substack{m\sim N/dp \\ \gcd(m,dq)=1}} \mu^2 (m)\\
& \qquad \qquad=   \left( \frac{\xi  \log(1/\zeta)}{\zeta(2)} + o(1) \right)  \prod_{\ell \mid dq}\(1+\frac{1}{\ell}\)^{-1}  \frac{N}{d}.
 \end{align*}
\end{lemma}

\begin{proof} We define
$$
U=  N/\tau(dq)^2 \mand V = N/(d\log (dq)).
$$ 

We first consider the part over primes $p  \leq U$. Applying Lemma~\ref{lem: SqFAP} to the inner sum, we obtain
\begin{align*} 
 \sum_{\substack{N^{\zeta} \le p \le U \\ \gcd(p,dq)=1}} & \sum_{\substack{m\sim N/dp \\ \gcd(m,dq)=1}} \mu^2 (m)\\
 &=   \frac{\xi}{\zeta(2)} 
 \prod_{\ell \mid dq}\(1+\frac{1}{\ell}\)^{-1}\frac{N}{d} 
\sum_{\substack{N^{\zeta} \le p \le  U  \\ \gcd(p,dq)=1}}\frac{1}{p} \\
&\qquad \qquad \qquad\quad  + O\(\sum_{\substack{ N^{\zeta} \le p \le U  \\ \gcd(p,dq)=1}} \(\frac{N}{dp}\)^{1/2}\tau(dq) \)  . 
 \end{align*}
We discard the coprimality condition and extend the summation over all primes $p \le U$. By the prime number theorem and partial summation, the error term is

\begin{align*} 
\sum_{\substack{ N^{\zeta} \le p \le U  \\ \gcd(p,dq)=1}} \(\frac{N}{dp}\)^{1/2}\tau(dq) & = 
N^{1/2}  d^{-1/2}\tau(dq) \sum_{\substack{ N^{\zeta} \le p \le U  \\ \gcd(p,dq)=1}} \frac{1}{p^{1/2}}\\
&\ll N^{1/2} d^{-1/2} \tau(dq)  \frac{U^{1/2}}{\log U} \ll \frac{N}{d^{1/2} \log N} 
 \end{align*}
since by the bound on the divisor function~\eqref{eq:tau} we have $U = N^{1+o(1)}$. 

Since   there are only $O(1)$ primes $p \mid dq$ with $p > N^{\zeta}$, using~\eqref{eq:Mer}, we now obtain 
\begin{align*} 
 \sum_{\substack{N^{\zeta} \le p \le U \\ \gcd(p,dq)=1}}\frac{1}{p}
&= \sum_{N^{\zeta} \le p \le U}\frac{1}{p} + O(N^{-\zeta})  \\
&  = \log \frac{\log N - 2\log \tau(dq) }{\zeta \log N}  +O\(\frac{1}{\log N}\).
 \end{align*}
Thus, using  the Mertens formula~\eqref{eq:Mer}, we
obtain
$$
 \sum_{\substack{N^{\zeta} \le p \le U \\ \gcd(p,dq)=1}}\frac{1}{p}
 =  \log (1/\zeta)+o(1), 
$$
 and thus
\begin{align*}
 & \sum_{\substack{N^{\zeta} \le p \le U \\ \gcd(p,dq)=1}}    \sum_{\substack{m\sim N/dp \\ \gcd(m,dq)=1}} \mu^2 (m) \\
 & \quad\qquad \quad  =  \( \frac{\xi  \log(1/\zeta)}{\zeta(2)} +o(1)\)  \prod_{\ell \mid dq}\(1+\frac{1}{\ell}\)^{-1}  \frac{N}{d} 
+ O\(\frac{N}{d^{1/2}\log N}\)\,.
\end{align*} 
We now  add the contribution from primes $U < p \le N$, which  we
 \begin{itemize}
 \item  put  in the error term;
\item further separate 
into two ranges   $U<p \le V$  and
$V<p \le N$;  
\item  replace $\mu^2 (m)$ with $1$ and abandon the condition $ \gcd(p,dq)=1$. 
\end{itemize}
Hence we derive
\begin{equation}
\label{eq:MEEE}
  \sum_{\substack{N^{\zeta} \le p \le N \\ \gcd(p,dq)=1}}    \sum_{\substack{m\sim N/dp \\ \gcd(m,dq)=1}} \mu^2 (m)  =  \MT + O\(\ET_1 + \ET_2 + \ET_3\) 
\end{equation} 
with the main term
$$
 \MT= 
\( \frac{\xi  \log(1/\zeta)}{\zeta(2)} +o(1)\)  \prod_{\ell \mid dq}\(1+\frac{1}{\ell}\)^{-1}  \frac{N}{d} , 
$$
where
$$
  \ET_1 = \frac{N}{d^{1/2}\log N}, \quad  \ET_2=  \sum_{ U < p \le V}  \sum_{\substack{m\sim N/dp \\ \gcd(m,dq)=1}} 1 , 
 \quad 
 \ET_3=  \sum_{\ V < p \le N}  \sum_{\substack{m\sim N/dp \\ \gcd(m,dq)=1}} 1  
$$
are  the  error terms, which we estimate separately. 

Note that  the range $U<p \le V$  can be empty, 
and thus $\ET_2=0$ in this case.

Since, trivially, $dq$ has at most $O\(\log (dq)\)$ prime divisors, by~\eqref{eq:Mer} 
we obtain  
$$
\prod_{\ell \mid dq}\(1+\frac{1}{\ell}\) \ll \log \log (dq) \ll \log \log N
$$
and thus under the condition  
 $d =o\((\log N)^2/\log \log N\)$, 
 we see that 
\begin{equation}
\label{eq:E1}
\ET_1 = o(\MT).
\end{equation}
 We now follow closely the proof of~\cite[Lemma~8]{HaSh}. 
 
 In the range $U < p \le V$ we apply Lemma~\ref{lem: AP bound}, 
 which yields
 \begin{equation}
\label{eq: mid range}
\ET_2 \ll    \frac{\varphi(dq)}{dq} \frac{N}{d} \sum_{U < p \le N/(d\log (dq))} \frac{1}{p}. 
 \end{equation} 
 Clearly 
\begin{equation}
\label{eq: prod l}
 \frac{\varphi(dq)}{dq} =  \prod_{\ell \mid dq}\(1-\frac{1}{\ell}\) \le  \prod_{\ell \mid dq}\(1+\frac{1}{\ell}\)^{-1}.
 \end{equation} 
 Using the Mertens formula~\eqref{eq:Mer} again, we obtain 
\begin{align*} 
\sum_{U < p \le V} \frac{1}{p}  
&= \log \frac{\log N - \log(d\log (dq))}{  \log N- \log \log \tau(dq)}  +O\(\frac{1}{\log N}\)\\
&  =   \log\( 1+ O\( \frac{\max\{\log(\tau(dq)),  \log(d \log (dq))\}}{\log N}\)\) \\
& \qquad  \qquad  \qquad  \qquad  \qquad  \qquad  \qquad  \qquad  \qquad +O\(\frac{1}{\log N}\) \\
& \ll  \frac{\max\{\log(\tau(dq)),  \log(d\log (dq))\}}{\log N}, 
 \end{align*}
which after substituting in~\eqref{eq: mid range} and recalling~\eqref{eq:tau} and~\eqref{eq: prod l}, implies 
\begin{equation}
\label{eq:E2}
\ET_2 = o(\MT).
\end{equation}

Finally in the range  $V < p \le N$,  we use the trivial bound
$$
\ET_3   \le \sum_{V < p \le N}  \sum_{\substack{m\sim N/dp \\ \gcd(m,dq)=1}} 1
\le N \sum_{V < p \le N}  \frac{1}{p} , 
$$
which, as in the proof of~\cite[Lemma~8]{HaSh}, together with the Mertens formula~\eqref{eq:Mer} implies 
\begin{equation}
\label{eq:E3}
\ET_3 \ll   N \frac{ \log \log N}{\log N} =  o(\MT).
\end{equation}
Substituting~\eqref{eq:E1},   \eqref{eq:E2} and~\eqref{eq:E3} in~\eqref{eq:MEEE}, we conclude the proof.
\end{proof}

We write for convenience  $\cB= \{n:~\gcd(n,q)=1\}$ and let 
$$
c_n = \begin{cases}1 &\text{if $p\mid n \Rightarrow p < N^{\zeta}$},\\
0 &\text{otherwise.}
\end{cases}
$$

We are now able to establish our main technical statement, which is an asymptotic formula
for the sum
\begin{equation}
\label{eq:Sum S}
 S^\sharp=\sum_{\substack{mnr  \in \cB \\ m, n \sim N}} b_r c_m\mu^2(mn),
 \end{equation} 
 where the summation is
only over $m$ and $n$ with square-free products $mn$. A similar 
 sum
$$
 S=\sum_{\substack{mnr  \in \cB \\ m, n \sim N}} b_r c_m
 $$
treated in the proof of~\cite[Lemma~8]{HaSh} can be deal with much simpler as
the variable are independent and we can write
$$
 S=\sum_{\substack{mnr  \in \cB \\ m, n \sim N}} b_r c_m
 = \sum_{\substack{m \in \cB \\ m\sim N}}  c_m 
 \sum_{\substack{n \in \cB \\ n \sim N}} 1
  \sum_{r \in \cB} b_r  .
$$

\begin{lemma}\label{lem:sums} For integers $N \ge q^{1/4} > 1$,  real $1/2 < \zeta \le 1$ 
and any finitely supported  sequence $b_r$, for the sum~\eqref{eq:Sum S}
we have
$$
 S^\sharp = C\(\frac{\xi N}{\zeta(2)}\)^2\prod_{\ell \mid q}\(1+ \frac{2}{\ell}\)^{-1} (1+\log \zeta +o(1)) \sum_{r \in \cB} b_r, 
$$
where 
$$
C=\prod_{\ell}\(1-\frac{1}{(\ell+1)^2}\).
$$
\end{lemma} 

\begin{proof}
Using the elementary identity~\eqref{inversion}, we have
$$
 S^\sharp = \sum_{\substack{m,n \in \cB \\ m,n \sim N}}  c_m\mu^2(m)\mu^2(n) \sum_{d \mid \gcd(n,m)} \mu(d)  \ \sum_{r \in \cB} b_r  \,.
$$

Switching summation and  using the multiplicativity of coefficients $c_m$, we obtain (using that $\mu^3(d)=\mu(d)$)
\begin{align*} 
 S^\sharp&=\sum_{\substack{d \leq \psi N\\ \gcd(d,q)=1}} \mu(d)c_d \sum_{\substack{m\sim N/d,\, \gcd(m,dq)=1 \\ 
n\sim N/d, \, \gcd(n,dq)=1}} \mu^2(m)\mu^2(n)c_m  \sum_{r \in \cB} b_r \\ 
&=\sum_{\substack{d \leq \psi N\\ \gcd(d,q)=1}} \mu(d)c_d S_1(d) S_2(d) \sum_{r \in \cB} b_r,
 \end{align*} 
 where 
$$
 S_1(d) = \sum_{\substack{m\sim N/d \\\gcd(m,dq)=1}} \mu^2(m)c_m  \mand  S_2(d) = \sum_{\substack{n\sim N/d \\ \gcd(n,dq)=1}} \mu^2(n).
$$
We set 
\begin{equation}
\label{eq: def D}
D =  \log N 
 \end{equation} 
and  estimate the contribution in $S$ from $d \ge D$ using  trivial
estimate 
$$
 |S_1(d)|, |S_2(d)| \le  N/d
 $$
 as $O(N^2/D)$. Thus we obtain 
\begin{equation}
\label{eq:SSSD}
 S^\sharp=\gamma \sum_{r \in \cB} b_r,
 \end{equation} 
 where 
\begin{equation}
\label{eq: gamma}
\gamma  = \sum_{\substack{d \le D\\ \gcd(d,q)=1}} \mu(d)c_d S_1(d) S_2(d)  + O(N^2/D).
 \end{equation} 

So we now assume $d \le D$. Then, by Lemma~\ref{lem: SqFAP} and the divisor bound, we have
$$S_2(d) = \(\frac{\xi}{\zeta(2)}+o(1)\)\prod_{\ell \mid dq}\(1+\frac{1}{\ell}\)^{-1} \frac{N}{d}.
$$ We now evaluate $S_1(d)$. We remark (since $\zeta > \frac12$) that
$$
\sum_{\substack{m\sim N/d \\\gcd(m,dq)=1}} \mu^2(m)c_m  =\sum_{\substack{m\sim N/d \\\gcd(m,dq)=1}}\mu^2(m) - 
\sum_{\substack{N^{\zeta} \le p\le N \\ \gcd(p,dq)=1}}\sum_{\substack{m\sim N/(dp) \\ \gcd(m,dq)=1}} \mu^2(m)\,.$$ 
Lemmas~\ref{lem: SqFAP} and~\ref{smooth} then give 
\begin{align*} 
\sum_{\substack{m\sim N/d \\\gcd(m,dq)=1}}&\mu^2(m)\\&  =  \(\frac{\xi}{\zeta(2)}+o(1)\)\prod_{\ell \mid dq}\(1+\frac{1}{\ell}\)^{-1} \frac{N}{d} + O\(\(N/d\)^{1/2} \tau(dq)\) 
 \end{align*} 
and 
$$
\sum_{\substack{N^{\zeta} \le p\le N \\ \gcd(p,dq)=1}}\sum_{\substack{m\sim N/(dp) \\ \gcd(m,dq)=1}} \mu^2(m) =
\(\frac{\xi \log(1/\zeta) }{\zeta(2)}+o(1)\)  \prod_{\ell \mid dq}\(1+\frac{1}{\ell}\)^{-1}  \frac{N}{d} ,
$$ 
respectively.
Thus subtracting, we derive
$$ 
S_1(d) = \(\frac{\xi}{\zeta(2)} (1+\log\zeta )+o(1)\)\prod_{\ell \mid dq}\(1+\frac{1}{\ell}\)^{-1}  \frac{N}{d}.
$$Hence for $d \le D$ we have
$$ 
S_1(d)S_2(d) = \(\frac{\xi N}{\zeta(2)d}\)^2\prod_{\ell \mid dq}\(1+\frac{1}{\ell}\)^{-2} (1+\log\zeta +o(1))
$$
 which after the substitution in~\eqref{eq: gamma} implies
 \begin{equation}
\label{eq: gamma asymp1}
\begin{split} 
\gamma = \(\frac{\xi N}{\zeta(2)}\)^2& \prod_{\ell \mid q}\(1+\frac{1}{\ell}\)^{-2} (1+\log\zeta +o(1))\\
&\sum_{\substack{d \leq D\\ \gcd(d,q)=1}}\frac{\mu(d)}{d^2}c_d\prod_{\ell \mid d}\(1+\frac{1}{\ell}\)^{-2}  
+ O(N^2/D).
\end{split} 
 \end{equation} 
Now,
\begin{align*} 
\sum_{\substack{d \leq D\\ \gcd(d,q)=1}}\frac{\mu(d)}{d^2}c_d&\prod_{\ell \mid d}\(1+\frac{1}{\ell}\)^{-2}\\
& = \sum_{\substack{d=1\\ \gcd(d,q)=1}}^{\infty} \frac{\mu(d)}{d^2}c_d\prod_{\ell \mid d}\(1+\frac{1}{\ell}\)^{-2}  + O(D^{-1}).    
 \end{align*} 
Recalling the definition of the coefficients $c_d$, we obtain
\begin{equation}
\label{eq: gamma asymp2}
\begin{split} 
\sum_{\substack{d \leq D\\ \gcd(d,q)=1}}\frac{\mu(d)}{d^2}c_d&\prod_{\ell \mid d}\(1+\frac{1}{\ell}\)^{-2}\\
& = \prod_{\substack{\ell \le N^{\zeta} \\ \gcd(\ell,q)=1}} \(1-\frac{1}{(\ell+1)^2}\) + O(D^{-1})    .
\end{split} 
 \end{equation} 
 Furthermore
\begin{align*}
\prod_{\substack{\ell \le N^{\zeta} \\ \gcd(\ell,q)=1}}& \(1-\frac{1}{(\ell+1)^2}\)  \\
&= \prod_{\substack{\ell \le N^{\zeta} \\ \ell\mid q}} \(1-\frac{1}{(\ell+1)^2}\)^{-1}
\prod_{\ell \le N^{\zeta}} \(1-\frac{1}{(\ell+1)^2}\) 
 \\
&= \prod_{\ell\mid q} \(1-\frac{1}{(\ell+1)^2}\)^{-1}\prod_{\ell}\(1-\frac{1}{(\ell+1)^2}\)(1+O(N^{-\zeta})).
\end{align*} 
Combining this with~\eqref{eq: gamma asymp2} and then substituting in~\eqref{eq: gamma asymp1}, 
we obtain 
\begin{align*}
\gamma = \(\frac{\xi N}{\zeta(2)}\)^2& \prod_{\ell \mid q}\(1+\frac{1}{\ell}\)^{-2} 
\prod_{\ell\mid q} \(1-\frac{1}{(\ell+1)^2}\)^{-1}\\
&\prod_{\ell}\(1-\frac{1}{(\ell+1)^2}\)
(1+\log\zeta +o(1))  
+ O(N^2/D).
\end{align*} 
Since
$$
\(1+\frac{1}{\ell}\)^2 \(1-\frac{1}{(\ell+1)^2}\)   = 1 +\frac{2}{\ell} , 
$$
we obtain 
\begin{align*}
\gamma = \(\frac{\xi N}{\zeta(2)}\)^2& \prod_{\ell \mid q}\(1 +\frac{2}{\ell} \)^{-1}\\
&\prod_{\ell}\(1-\frac{1}{(\ell+1)^2}\)
(1+\log\zeta +o(1))  
+ O(N^2/D).
\end{align*} 
By the version of the Mertens formula~\eqref{eq:Mer}, similarly to~\eqref{eq:phi} we have 
\begin{equation}
\label{eq: prod 2l}
 \prod_{\ell \mid q}\(1 +\frac{2}{\ell} \)^{-1} \ge  \prod_{\ell \mid q}\(1 +\frac{1}{\ell} \)^{-2} 
 \gg (\log \log q)^{-2}.
\end{equation} 

Therefore with the above choice~\eqref{eq: def D} of $D$, we obtain 
$$
\gamma = \(\frac{\xi N}{\zeta(2)}\)^2 \prod_{\ell \mid q}\(1 +\frac{2}{\ell} \)^{-1}\prod_{\ell}\(1-\frac{1}{(\ell+1)^2}\)
\(1+\log\zeta +o(1) \), 
$$
which together with~\eqref{eq:SSSD} concludes the proof. 
\end{proof}
 
\begin{remark} For prime $q$ the proofs of Lemmas~\ref{smooth} and~\ref{lem:sums}   simplify quite significantly and the factors depending on $q$ all become  equal to $1$. 
 \end{remark}

 Taking $\zeta=1$ in  Lemma~\ref{lem:sums} (or proceeding in a  similar manner, 
 but using only the estimates on $S_2(d)$), we obtain: 
 
  \begin{cor}\label{cor:productsqrfree} We have,
 $$\sum_{\substack{mnr  \in \cB \\ m, n \sim N}} b_r \mu^2(mn) = (C+o(1))\(\frac{\xi N}{\zeta(2)}\)^2\prod_{\ell\mid q}\(1+\frac{2}{\ell}\)^{-1}  \sum_{r \in \cB} b_r.$$
 \end{cor}
 
 \begin{remark} Removing the summation over $r$ in Corollary~\ref{cor:productsqrfree}, that is, taking the sequence $b_r$ which is supported only on $r=1$, 
 we obtain an asymptotic formula  for the number of products $mn$ being square-free with $m,n$ of same size and coprime to a fixed number. \end{remark}

\subsection{Some  sieving result}
\label{sec:sumprod-AP}

We now specify
\begin{equation}
\label{eq: zeta}
\zeta = \rho(1 + \varepsilon) = \frac{1 + \varepsilon}{e^{1/2}}, 
\end{equation} 
where $\rho$ is given by~\eqref{eq: rho}, and 
derive  an analogue of~\cite[Lemma~9]{HaSh}.

\begin{lemma}\label{lem:sieve} 
Let be $N \ge q^{1/4} > 1$ integers  and let  $\varepsilon >0$ be a sufficiently small
fixed real number.   
Suppose that $\cA \subseteq \cB$ is a set such that  for a    finitely supported  sequence $b_r$ and 
some $\lambda > 0$ and $\eta > 0$, we have 
\begin{equation}\label{Lambda}
\sum_{\substack{mnr  \in \cA\\ m, n \sim N}} a_n b_r \mu^2(mn)
=\lambda
\sum_{\substack{mnr \in \cB\\ m, n \sim N}} a_n b_r \mu^2(mn)
+ O(\lambda x^{1 - \eta})
\end{equation}
for any sequence $a_n = O(1)$. 
Then for $\zeta$ given by~\eqref{eq: zeta}
\begin{align*}
\sum_{\substack{mnr  \in \cA\\ m, n \sim N}} & b_r c_n c_m \mu^2(mn)\\
& \ge C\(\frac{\xi N}{\zeta(2)}\)^2\prod_{\ell \mid q}\(1 +\frac{2}{\ell} \)^{-1}(2\log(1+\varepsilon)+o(1)) \sum_{r \in \cB} b_r\\
& \qquad \qquad \qquad \qquad \qquad \qquad \qquad \qquad \qquad \quad +
O(\lambda x^{1 - \eta})\,.
\end{align*}
\end{lemma}

\begin{proof} Generally, we follow very closely the argument of the proof of~\cite[Lemma~9]{HaSh}, 
however here we use Lemma~\ref{lem:sums}  and Corollary~\ref{cor:productsqrfree} 
instead of~\cite[Lemmas~6 and~8]{HaSh} in the corresponding place.

Using the  observation of Balog~\cite{Bal}, we have
$$
\sum_{\substack{mnr  \in \cA\\ m, n \sim N}}  b_r c_n c_m\mu^2(mn) \ge E - F
$$
where
$$
E = \sum_{\substack{mnr  \in \cA\\ m, n \sim N}}  b_r c_m\mu^2(mn) \,, \qquad F = \sum_{\substack{mnr  \in \cA\\ m, n \sim N}}  b_r h_n\mu^2(mn)\,,
$$
and $h_n = 1-c_n$.  By the assumption~\eqref{Lambda}, we have
$$
E = \lambda \sum_{\substack{mnr  \in \cB \\ m, n \sim N}} b_r c_m\mu^2(mn) +  O(\lambda x^{1 - \eta}) 
$$ 
and 
$$F = \lambda \sum_{\substack{mnr  \in \cB \\ m, n \sim N}} b_r h_n\mu^2(mn) +  O(\lambda x^{1 - \eta}).$$ 
Now, by Lemma~\ref{lem:sums}  
$$
\sum_{\substack{mnr  \in \cB \\ m, n \sim N}} b_r c_m\mu^2(mn) = C\(\frac{\xi N}{\zeta(2)}\)^2\prod_{\ell \mid q}\(1 +\frac{2}{\ell} \)^{-1}(1+\log \zeta+o(1)) \sum_{r \in \cB} b_r \,
$$ 
while by a combination of  Lemma~\ref{lem:sums}  and Corollary~\ref{cor:productsqrfree}
$$\sum_{\substack{mnr  \in \cB \\ m, n \sim N}} b_r h_n\mu^2(mn) =  C\(\frac{\xi N}{\zeta(2)}\)^2\prod_{\ell \mid q}\(1 +\frac{2}{\ell} \)^{-1}(-\log \zeta+o(1)) \sum_{r \in \cB} b_r.$$ After summation, we obtain 
$$E-F \geq  C\(\frac{\xi N}{\zeta(2)}\)^2\prod_{\ell \mid q}\(1 +\frac{2}{\ell} \)^{-1}(1+2\log \zeta+o(1)) \sum_{r \in \cB} b_r.$$ Remarking that $1+2\log \zeta=2\log(1+\varepsilon)$ we conclude the proof.
 \end{proof}
 
In the special case where $\cA = \cA_{a,q}(x)$ is the set of integers $k \in [x,2x]$ 
with $k \equiv a  \pmod q$, using that $2  \log(1+\varepsilon) > \varepsilon$ if $\varepsilon >0$ is a sufficiently small,
we derive:

\begin{cor}\label{cor:AP} 
Assume that the condition of Lemma~\ref{lem:sieve} 
holds with $x_0(\varepsilon)$ for the set $\cA =  \cA_{a,q}(x)$
with $\lambda = 1/ \varphi(q)$,  where  $x_0(\varepsilon)$ depends
only on $\varepsilon$ and is sufficiently large. 
Then, for a    finitely supported  sequence of positive real numbers $b_r$ and 
some  $\eta > 0$, we have 
\begin{align*}  
\sum_{\substack{mnr  \in  \cA_{a,q}(x)\\ m, n \sim N}}   & b_r  c_n c_m \mu^2(mn)\\& 
\ge \varepsilon C\(\frac{\xi N}{\zeta(2)}\)^2\prod_{\ell \mid q}\(1 +\frac{2}{\ell} \)^{-1} \sum_{r \in \cB} b_r
+ O(x^{1 - \eta} q^{-1})\,.
\end{align*}
\end{cor}

\subsection{Products in arithmetic progressions}

We now define  the following parameter
\begin{equation}
\label{eq: nu}
\nu= \rf{1/\varepsilon}.
\end{equation}
For a given $q$, we consider the set of integers $r$ that are 
products of $13$ distinct primes of the form 
\begin{equation}
\label{eq:set r}
r = \ell_1\ldots \ell_{12} s\mand
\gcd(r,q)=1, 
\end{equation}
where 
\begin{equation}
\label{eq:primes}
\ell_1, \ldots, \ell_{12} \sim q^{1/8}, \qquad 
 s \sim q^{1/\nu},
\end{equation}
and let $b_r$ be the characteristic 
function of this set. We note that $b_r$ is supported on the interval $[R, \psi^{13} R]$ with $R=q^{3/2 +1/\nu}$.

As before, we also define by $\cA_{a,q}(x)$ the set of integers $k \in [x,2x]$ 
with $k \equiv a  \pmod q$.

Next, repeating word-by-word the argument of the proof of~\cite[Lemma~14]{HaSh}, 
but using Lemma~\ref{lem:CharSquarfree}  instead of the classical Burgess 
bound as in~\cite{HaSh},  we now show that for any fixed sufficiently small $\varepsilon> 0$ 
the conditions of Lemma~\ref{lem:sieve}  are satisfied for  the set 
$\cA = \cA_{a,q}(x)$ and the choice of $b_r$ with $N=q^{1/4 + \varepsilon}$
upon writing $x=N^2R$.  

\begin{lemma}  
\label{lem:Cong} Let $\varepsilon > 0$ be sufficiently small, $q> 1$ and $N=q^{1/4 + \varepsilon}$.    
Suppose that the sequence  $b_r$ is the characteristic 
function of the set  defined by~\eqref{eq:set r} and~\eqref{eq:primes}.
Then for integers $a$ and $q$ with $\gcd(a,q)=1$ and such that $q$ is 
cube-free  we have
$$
\sum_{\substack{mnr  \in \cA_{a,q}(x)\\ m, n \sim N}}
a_n b_r \mu^2(mn) =  \frac{1}{\varphi(q)}\sum_{\substack{mnr  \in \cB\\ m, n \sim N}}
a_n b_r \mu^2(mn)
 + O\(q^{-1} x^{1-\eta}\)
$$
with $\eta = \varepsilon^4$,
 $R=q^{3/2 + 1/\nu}$, where $\nu$ is as in~\eqref{eq: nu}, and $x=N^2R$, and any sequence $a_n$ 
 satisfying $|a_n| \le n^{o(1)}$. 
\end{lemma}

\begin{proof} 
We start with the observation that  if $b_r \ne 0$ and $m, n \sim N$ 
then due to the choice of our parameters we always have 
$$
mnr  \in [N^2R, \psi^{15}N^2 R] \subseteq [x,2x].
$$
In particular, if $b_r \ne 0$ and $m, n \sim N$  then the condition 
$mnr  \in  \cA_{a,q}(x)$ is equivalent to the 
congruence $mnr  \equiv a \pmod q$ and the condition 
$mnr  \in  \cB$ is merely equivalent to $\gcd(mn,q)=1$. 

Let 
$$
S = \sum_{\substack{mnr  \in \cA_{a,q}(x)\\ m, n \sim N}}
a_n b_r \mu^2(mn).
$$

 Using the orthogonality 
of characters we write 
$$
\rS = \sum_{\substack{mnr  \in \cB\\ m, n \sim N}}
a_n b_r \mu^2(mn)\frac{1}{\varphi(q)} \sum_{\chi\in \cX_q} \chi(mnr a^{-1}).
$$
Using the equation~\eqref{inversion}, we  write
\begin{align*}
\rS = \frac{1}{\varphi(q)} \sum_{\chi\in \cX_q}& \sum_{\substack{r \in \cR \\ m,n \sim N}}  a_n \chi(mn)\mu^2(m)\mu^2(n)\\
& \qquad  \sum_{d \mid \gcd(n,m)} \mu(d)  \ \sum_{r \in \cR} \chi(r) b_r \chi(a^{-1})\,. 
\end{align*}
Thus, rearranging the summation, we obtain
\begin{equation}
\label{eq:SSd}
\rS= \sum_{d \leq \psi N} \mu(d)\rS_d, 
\end{equation}
where
$$
\rS_d = \frac{1}{\varphi(q)} \sum_{\chi\in \cX_q}   \sum_{\substack{m\sim N/d,\, \gcd(m,d)=1 \\ 
n\sim N/d, \, \gcd(n,d)=1}}  \chi(mn)\mu^2(m)\mu^2(n)a_{nd}  \sum_{r \in \cR} \chi(r)\chi(d^2a^{-1}). 
$$
In particular $S_d = 0$ unless $\gcd(d,q)=1$, in which case, 
using the orthogonality of characters again, we see that 
$$
\rS_d =   \sum_{\substack{mnr  \in \cA_{ad^{-2},q}(x)\\ m,n\sim N/d, \, \gcd(mn,d)=1 }}
 a_{nd}\mu^2(m)\mu^2(n)  b_r.
$$
Since when $m$ and $n$ are fixed, the value of  $r$ is uniquely defined modulo $q$ and thus can take $O(R/q)$ possible values (recall that $R\ge q$), we have 
the trivial  estimate
\begin{equation}
\label{eq:Triv Sd}
\rS_d \le 
N^{2+o(1)}R d^{-2} q^{-1}= x^{1+o(1)} d^{-2} q^{-1}.
\end{equation}

On the other hand, separating the contribution from the principal character in~\eqref{eq:SSd}, we also write 
\begin{equation}
\label{eq:SdMdEd}
\rS_d = \MT_d + O\(\ET_d\), 
\end{equation}
with the main term
$$
 \MT_d = \frac{1}{\varphi(q)} \sum_{\substack{mnr  \in \cB\\ m, n \sim N/d,\, \gcd(mn,d)=1 }}
a_{nd} b_r \mu^2(m)\mu^2(n) 
$$
and the error term
$$
 \ET_d = \frac{1}{\varphi(q)} \sum_{\chi\in \cX_q^*}  \left|\ET_d(\chi) \right|,
 $$
where 
$$
\ET_d(\chi) =  \sum_{\substack{m\sim N/d,\, \gcd(m,d)=1 \\ 
n\sim N/d, \, \gcd(n,d)=1}}  \chi(mn)\mu^2(m)\mu^2(n)a_{nd}  \sum_{r \in \cR} \chi(r)\chi(d^2a^{-1}) .
$$
We now use some parameter $D$ and use~\eqref{eq:Triv Sd} for $d > D$ and use~\eqref{eq:SdMdEd}
otherwise. Together with~\eqref{eq:SSd}, this leads to the asymptotic formula 
\begin{equation}
\label{eq:SME-1}
\rS=    \sum_{d \leq \psi N} \mu(d)  \MT_d + O\(\ET +  x^{1+o(1)} D^{-1} q^{-1}\),
\end{equation}
where 
$$
\ET =  \sum_{d \leq D}  \left|\ET_d(\chi) \right|.
$$

Using the trivial upper bound $\MT_d  \le x^{1+o(1)} d^{-2} q^{-1}$, which is similar to~\eqref{eq:Triv Sd}, 
we obtain
\begin{align*}   \sum_{d \le D } \mu(d)  \MT_d & =     \sum_{d \leq \psi N}  \mu(d) \MT_d + O\( x^{1+o(1)} D^{-1} q^{-1}\)\\
& = \frac{1}{\varphi(q)}\sum_{\substack{mnr  \in \cB\\ m, n \sim N}}
a_n b_r \mu^2(mn)
+ O\(  x^{1+o(1)} D^{-1} q^{-1}\), 
\end{align*}
which together with~\eqref{eq:SME-1} implies
\begin{equation}
\label{eq:SME-2}
\rS=   \frac{1}{\varphi(q)}\sum_{\substack{mnr  \in \cB\\ m, n \sim N}}
a_n b_r \mu^2(mn)
 + O\(\ET +  x^{1+o(1)} D^{-1} q^{-1}\). 
\end{equation}
Hence it remains to estimate $\ET$ which we do by estimating individually $\ET_d$ for $d\le D$. 
From now on, we fix 
$$
D=x^{\varepsilon^2}.
$$

For a real $\omega> 0$ we consider the character sums over 
primes
$$
V_\omega(\chi)  = \sum_{\ell\sim q^{\omega}} \chi(\ell), 
$$
which we use with $\omega =1/8$ and $\omega=1/\nu$. 
We also consider 
the weighted sums
$$
W_d(\chi) = 
\sum_{\substack{m \sim N/d \\ \gcd(m,d)=1}} \sum_{\substack{n \sim N/d \\ \gcd(n,d)=1}}  \sum_{v \in \cV} a_{nd}  \chi(mnv)\mu^2(m)\mu^2(n), 
$$
where $v$ runs through  the set $\cV$ of $q^{1/2 + o(1)}$ products 
$v= \ell_1 \ell_2 \ell_3 \ell_4$ as in~\eqref{eq:set r}. Thus, we can write
\begin{equation}
\label{eq:fE2}
\vert \ET_d \vert = \frac{1}{\varphi(q)} \sum_{\chi\in \cX_q^*}\left|V_{1/8}(\chi)\right|^8  \left|V_{1/\nu}(\chi)\right|\left|W_d(\chi) \right|. 
\end{equation}

We now collect the 
currently available information 
about the sums $V_{1/8}(\chi)$, $V_{1/\nu}(\chi)$ and $W_d(\chi)$.

First, since $N=  q^{1/4 + \varepsilon}$, for a sufficiently  small $\varepsilon>0$ and the above choice of $D$ we 
have $N/d > (dq)^{1/4 +\varepsilon/2}$. Hence, by Corollary~\ref{cor:CharSquarfree-eps}  applied to the character 
$\chi \chi_{0}^{(d)}$ where $\chi \in  \cX_q^*$ and $\chi_{0}^{(d)}$ is the trivial character modulo $d$, 
we have
\begin{equation}
\begin{split}
\label{eq:indiv}
\max_{\chi \in \cX_q^*} \left|W_d(\chi)\right|  & 
\le \frac{N^{1+o(1)}}{d} q^{1/2} \max_{\chi \in \cX_q^*} 
\left|\sum_{\substack{m \sim N/d \\ \gcd(m,d)=1}}   \chi(m)\mu^2(m)\right| \\
& \ll  \left(\frac{N}{d}\right)^{2-c_0\varepsilon^2}q^{1/2}
\end{split}
\end{equation}
with some absolute constant $c_0> 0.$ 

We  also have the inequalities
\begin{equation}
\label{eq:aver1}
 \sum_{\chi \in \cX_q} \left|V_{1/8}(\chi) \right|^{16}  \ll q^{2}, 
\qquad \sum_{\chi \in \cX_q} \left|V_{1/\nu}(\chi) \right|^{2\nu} \ll q^{2},
\end{equation}
and 
\begin{equation}
\label{eq:aver2}
 \sum_{\chi\in \cX_q}   \left|W_d(\chi)\right| ^2
\le  (N/d)^2 q^{3/2+o(1)}\(1 + (N/d)^2q^{-1/2}\) , 
\end{equation}
implied by  Lemma~\ref{lem:Aver}.  Since for the above choice of parameters we have $(N/d)^2 \ge q^{1/2}$
(provided $\varepsilon$ is small enough) 
the inequality~\eqref{eq:aver2} simplifies as
\begin{equation}
\label{eq:aver3}
 \sum_{\chi\in \cX_q}   \left|W_d(\chi)\right| ^2
  \le (N/d)^4 q^{1+o(1)}. 
\end{equation}
We now write $ \left|W_d(\chi) \right| = \left|W_d(\chi) \right|^{1/\nu}  \left|W_d(\chi) \right|^{1-1/\nu}
$ and apply~\eqref{eq:indiv}, deriving from~\eqref{eq:fE2}
\begin{equation}
\label{eq:fE3}
\begin{split}
\vert \ET_d \vert 
\le \frac{1}{\varphi(q)} \(\left(\frac{N}{d}\right)^{2-c_0\varepsilon^2}q^{1/2}\)^{1/\nu} \sum_{\chi\in \cX_q^*}&\left|V_{1/8}(\chi)\right|^8  \left|V_{1/\nu}(\chi)\right| \left|W_d(\chi) \right|^{1-1/\nu}. 
\end{split}
\end{equation}
Finally, since 
$$
\frac{1}{2} + \frac{1}{2\nu} + \frac{\nu-1}{2\nu}  = 1
$$
 by the H{\"o}lder inequality, applied to the sum in~\eqref{eq:fE3}, and extending the summation to 
 all $\chi \in \cX_q$, we obtain 
\begin{align*}
\vert \ET_d \vert \le \frac{1}{\varphi(q)} & \(\left(\frac{N}{d}\right)^{2-c_0\varepsilon^2}q^{1/2}\)^{1/\nu} 
\( \sum_{\chi \in \cX_q} \left|V_{1/8}(\chi) \right|^{16}\right)^{1/2} \\
& \qquad \(\sum_{\chi\in \cX_q} \left|V_{1/\nu}(\chi)\right|^{2\nu}\)^{1/(2\nu)}
\(\sum_{\chi\in \cX_q} \left|W_d(\chi) \right|^2\)^{(\nu-1)/(2\nu)}. 
\end{align*}  
Recalling~\eqref{eq:aver1} and~\eqref{eq:aver3},  we derive
\begin{align*}
\vert \ET_d \vert & \le \frac{1}{\varphi(q)}   \(\left(\frac{N}{d}\right)^{2-c_0\varepsilon^2}q^{1/2}\)^{1/\nu}  q^{1+1/\nu} \((N/d)^4 q^{1+o(1)}\)^{\frac{\nu-1}{2\nu}}\\
&= \frac{(N/d)^2}{\varphi(q)} q^{3/2+1/\nu + o(1)} (N/d)^{-c_0\varepsilon^2/\nu} =  \frac{x}{d^2}\varphi(q)^{-1}(N/d)^{-c_0\varepsilon^2/\nu+o(1)}  .
\end{align*}  
Noticing that $\varepsilon \geq 1/\nu$ and $d\leq x^{\varepsilon^2}$, we derive  
$$
\ET =  \sum_{d \leq D}  |\ET_d| \ll \frac{x^{1-\eta}}{\varphi(q)}
$$ 
for some $\eta= \varepsilon^4$ and a sufficiently small $\varepsilon$,  which together with~\eqref{eq:SME-2} 
concludes the proof. 
\end{proof}

\section{Proofs of Main Results}
 
\subsection{Proof of Theorem~\ref{thm:Malpha}}

\subsubsection{Cube-free moduli}

Here we always assume that $q$ is cube-free.

We fix some sufficiently small $\varepsilon > 0$.

Let $\rho$, $\zeta$ and  $\nu$ be as in~\eqref{eq: rho}, \eqref{eq: zeta} and~\eqref{eq: nu}, respectively,  and let 
$$
  \beta = \rho/4 =  \frac{1}{4e^{1/2}}.
$$
We also choose $N$, $R$ and $x$ as in  Lemma~\ref{lem:Cong} and  remark that    
$$
N^{\zeta} = q^{(1/4 + \varepsilon)\rho(1+\varepsilon)}  = q^{\beta + 5\beta \varepsilon + \varepsilon^2}
\le  q^{\beta +  \varepsilon}
$$
provided that $\varepsilon$ is small enough.

Finally,  we define $\cK$ as the following {\it multiset\/}  
\begin{equation}\label{defk}
\cK = \{k=mn~:~ m, n \sim N, \ \mu^2(mn)=1, \ p \mid mn \Rightarrow p < N^{\zeta}\}, 
\end{equation}
where the integers $k$ are counted with multiplicity in $\cK$.

For integers $a$ and $q$ with $\gcd(a,q)=1$ and such that $q$ is 
cube-free we consider 
the number $T$ of solutions to the congruence 
\begin{equation}
\label{eq:rk cong}
kr\equiv a \pmod q
\end{equation}
where 
$k\in \cK$, where the multiset $\cK$ is defined by~\eqref{defk} and 
$r$ is  defined by~\eqref{eq:set r} and~\eqref{eq:primes}.

By Lemma~\ref{lem:Cong} and then by Corollary~\ref{cor:AP}  we see 
that 
\begin{equation}
\begin{split}
\label{eq:T asymp}
T & =     \sum_{\substack{kr \in \cA_{a,q}(x)\\ k\in \cK}} b_r   
=   \frac{1}{\varphi(q)}\sum_{\substack{kr  \in \cB\\  k \in \cK}}
 b_r 
 + O\(q^{-1} x^{1-\eta}\)
\\
& \ge \varepsilon C\frac{1}{\varphi(q)} \(\frac{\xi N}{\zeta(2)}\)^2\prod_{\ell \mid q}\(1 +\frac{2}{\ell} \)^{-1} \sum_{r \in \cB} b_r
+ O(x^{1 - \eta} q^{-1})\,. 
\end{split}
\end{equation}

By the prime number theorem there are $q^{3/2+ 1/\nu+o(1)}$ values of $r$ given 
by~\eqref{eq:set r} and~\eqref{eq:primes} and for each of 
them $q^{3/2 + 1/\nu} \ll r \ll q^{3/2 + 1/\nu}$.
Hence, for a sufficiently small $\varepsilon > 0$, after simple calculations,
using~\eqref{eq:phi} and also~\eqref{eq: prod 2l}, 
we obtain from~\eqref{eq:T asymp} that
\begin{equation}
\label{eq:T large}
T \ge N^2R q^{-1+o(1)} =  xq^{-1+o(1)}.
\end{equation}

We see from the definition of the sets of $r$   and $k$ that if $kr $
is not square-free then it is divisible by a square of a prime $\ell \ge q^\kappa$ where $\kappa = \min\{1/8, 1/\nu\}$. 
Together with $kr \in  \cA_{a,q}(x)$ this puts the product $kr\le x$ in a prescribed  arithmetic progression modulo $q\ell^2$.  Thus there are at most $x/q\ell^2$ positive integers $t$ in any such progression. 
Summing over all $\ell \ge q^\kappa$ (and ignoring the primality constraint) we obtain at most 
$$
\sum_{\ell \ge q^\kappa}\frac{x}{q\ell^2} \ll xq^{-1-\kappa}
$$
such values of $t$. From the classical bound on the divisor function, see~\cite[Equation~(1.81)]{IwKow}, 
we infer that  each of  such $t$ leads to at most $t^{o(1)} = q^{o(1)}$ possible triples $(m,n,r)$ with 
$mn =k \in \cK$. Comparing this with (\ref{eq:T large}), out of $T$ solutions to~\eqref{eq:rk cong}  at most $Tq^{-\kappa+o(1)}$ are 
not square-free.
Since $mnr \le x = q^{2+2 \varepsilon +  \frac{1}{\nu}}$  we have 
\begin{equation}
\label{eq: prelim}
M_{1/(4 e^{1/2})+\varepsilon}^*(q) \ll   q^{2 + 3 \varepsilon}, 
\end{equation}
and  the result follows. 
Indeed, assuming that it fails, we see that there is $\varepsilon_0$ and $\delta_0$ such that 
$$
M_{1/(4 e^{1/2})+\varepsilon_0}^*(q) \gg   q^{2 + \delta_0}.
$$
Then taking $\varepsilon = \min\{\varepsilon_0,  \delta_0/4\}$ we obtain a contradiction
with~\eqref{eq: prelim}.

\subsubsection{Arbitrari moduli}

We can derive a result for $q$ non cube-free following the same lines.

 We define the parameters $N=q^{1/3+\varepsilon}$ and $\nu$ as in~\eqref{eq: nu}. We consider the set of integers $r$ that are 
products of $9$ distinct primes of the form 
\begin{equation}
\label{eq:set rbis}
r = \ell_1\ldots \ell_{8} s\mand
\gcd(r,q)=1, 
\end{equation}
where 
\begin{equation}
\label{eq:primesbis}
\ell_1, \ldots, \ell_{8} \sim q^{1/6}, \quad 
 s \sim q^{1/\nu}.
\end{equation}
As before,  $\cK$ we have  
\begin{equation}\label{defkbis}
\cK = \{k=mn~:~ m, n \sim N, \ \mu^2(mn)=1, \ p \mid mn \Rightarrow p < N^{\zeta}\}, 
\end{equation}
where the integers $k$ are counted with multiplicity in $\cK$.

For integers $a$ and $q$ with $\gcd(a,q)=1$, we similarly define
the number $T'$ of solutions to the congruence 
$$
kr\equiv a \pmod q
$$
where 
$k\in \cK$ with  the multiset $\cK$ is defined by~\eqref{defkbis} and 
$r$ is  defined by~\eqref{eq:set rbis} and~\eqref{eq:primesbis}. For this new set of parameters, 
applying Corollary~\ref{cor:CharSquarfree-eps} in the   case of arbitrary $q$ and following exactly the same lines, we can show that the conditions of Lemma \ref{lem:sieve} are fullfilled and prove an exact analogue of Lemma~\ref{lem:Cong}. The proof goes then exactly as in the case of cube-free moduli.

\subsection{Proof of Theorem~\ref{thm:Malpha-p-AA}}

We fix some integer $n  > 2/\alpha$ and
reals   $1/4 > \varepsilon, \delta > 0$  and  define
$$
\beta = 1-  2\delta/n +\varepsilon  
 \mand k = \rf{\beta/\alpha} 
$$

Clearly it is enough to prove Theorem~\ref{thm:Malpha-p-AA} for all but $Q^{o(1)}$ primes in
dyadic intervals $[Q/2,Q]$.

We   further denote
$$
T =\fl{(Q/2)^{2/n}}  \mand W = \fl{(Q/2)^{\beta}}
$$
 and define the sets
\begin{itemize}
\item $\cS$ as the set of square-free integers $s \le T$; 
\item $\cU$ as the set of  products $u = \ell_1\ldots  \ell_k $ of  $k$ distinct primes $ \ell_1,\ldots,  \ell_k  \in [0.5W^{1/k}, W^{1/k}]$.
\end{itemize}

We note that for the above definition we have
$$
n \le  \left\{\frac{2 \log Q}{\log T}\right\} = n +O(1/\log Q).
$$
Hence we have $\gamma \ll 1/\log Q$ in the conditions of  Lemma~\ref{lem:CharSquarfree-AA} 
and thus, recalling that $\delta<1/4$, we have 
\begin{equation}
\label{eq:bound p}
\max_{\chi\in\cX_p^*}
\left|S^\sharp_\chi(T)\right|\le T^{1-\delta}, 
\end{equation} 
for all but $Q^{4\delta+o(1)}$ primes  $p \in  [Q/2, Q]$.

Hence, 
we fix a prime $p \in  [Q/2, Q]$ for which the bound~\eqref{eq:bound p} holds.

Clearly  products $suv$ with $(s,u,v) \in \cS\times \cU\times \cU$ are $p^{\alpha}$-smooth
(as $k$ is chosen to satisfy $\beta/k < \alpha$), 
but generally speaking may not be  square-free. We now claim that there are many products of these type in every  reduced class modulo $p$. Then we show that at least one of these representatives  is square-free.  

Indeed, let us 
fix some integer $a$ with $\gcd(a,p) = 1$ and let $N$ be the number of solutions to
\begin{equation}
\label{eq:cong}
suv\equiv a \pmod p, \qquad (s,u,v) \in \cS\times \cU\times \cU.
\end{equation}

To show that $N > 0$ and thus prove the claim, for a real $x$, 
  as usual,  we denote by  $\pi(x)$   the number of primes $\ell \le x$.

Certainly the asymptotic formula 
$$
\# \cS \sim \frac{6}{\pi^{2}} T
$$ 
for the cardinality of $\cS$ 
is well known, however it is  quite enough for us to use the trivial bounds 
$$
T \ge \# \cS \ge \pi(T).
$$
We also use that 
$$
\#\cU  =\binom{\pi\(W^{1/k}\)}{k}. 
$$
It now follows from the prime number theorem that 
\begin{equation}
\label{eq:Card}
\# \cS = T^{1+o(1)} \mand  \#\cU = W^{1+o(1)}.
\end{equation} 

Using the orthogonality of characters we express the number of solutions~\eqref{eq:cong} as
$$
N= \sum_{(s,u,v) \in \cS\times \cU\times \cU} \frac{1}{p-1}\sum_{\chi\in \cX_p} \chi(suva^{-1}).
 $$
Changing the order of summation and using the multiplicativity of characters, we now obtain
$$
N =\frac{1}{p-1}\sum_{\chi\in \cX_p} \chi(a^{-1})  \(\sum_{u \in \cU}  \chi(u) \)^2 \sum_{s\in \cS}  \chi(s) .
$$
Now, separating the contribution from the principal character, we derive
\begin{equation}
\label{eq:T and R}
N =\frac{\#\cS\( \#\cU \)^2}{p-1}+ \frac{1}{p-1}R,
\end{equation} 
where 
$$
R = \sum_{\chi\in \cX_p^*} \chi(a^{-1}) \(\sum_{u \in \cU}  \chi(u) \)^2
S^\sharp_\chi(T)
$$
and $S^\sharp_\chi(T)$ is given by~\eqref{eq:sum s-f}. 
Since $p$ is chosen to satisfy the bound~\eqref{eq:bound p}, we have 
\begin{equation}
\label{eq:R}
|R| \le  T^{1-\delta} \sum_{\chi\in \cX_p^*} \left| \sum_{u \in \cU}  \chi(u) \right|^2 .
\end{equation} 
Furthermore, using the  orthogonality property~\eqref{eq:orth}, 
we obtain
$$
\sum_{\chi\in \cX_p^*} \left| \sum_{u \in \cU}  \chi(u) \right|^2
  \le \sum_{\chi\in \cX_p} \left| \sum_{u \in \cU}  \chi(u) \right|^2
 = (p-1) \#\cU.
$$
Recalling (\ref{eq:R}), we obtain
$$
|R|\le   T^{1-\delta}p^{1+o(1)} \#\cU, $$
which after substitution in~\eqref{eq:T and R} and then using~\eqref{eq:Card} gives
\begin{align*}
N & =\frac{\#\cS\( \#\cU \)^2}{p-1}
+ O\(T^{1-\delta} \#\cU\)\\
&= \frac{\#\cS\( \#\cU \)^2}{p-1}\(1+ O\(pT^{-\delta}W^{-1}\)\)
&= \frac{\#\cS\( \#\cU \)^2}{p-1}\(1+ O\(p^{-\varepsilon}\)\).
\end{align*}

It remains to show that out $N$ such products $suv$ satisfying~\eqref{eq:cong} 
we can find at least one square-free.

Similarly to the argument of the proof of Theorem~\ref{thm:Malpha} we note  that if $suv$
is not square-free that it is divisible by a square of a prime $\ell \in  (0.5W^{1/k}, W^{1/k}]$
and thus out of $N$ solutions to~\eqref{eq:cong}  at most $TW^{2-1/k}p^{-1+o(1)}$ are 
not square-free. We now conclude that for a sufficiently large $p$ at least one of the products 
$$
suv \le  TW^2 \le Q^{2+2(1-2\delta)/n  +2\varepsilon}
$$
satisfying~\eqref{eq:cong}, is  square-free. We recall that this holds for all but $Q^{4\delta+o(1)}$ primes  $p \in  [Q/2, Q]$.

Because $n$ can be chosen arbitrary large and  while $\varepsilon$ and $\delta > 0$ can be chosen arbitrary
small, the result now follows.

\subsection{Proof of Theorem~\ref{thm:M1}}

Let us fix some sufficiently small $\varepsilon > 0$ and set 
\begin{equation}
\label{eq:DhL}
D = p^{\varepsilon/4}, \qquad h = p-1, \qquad L = p^{1/4+\varepsilon}.
\end{equation}

Let $N_{a,p}^\sharp(L,h)$ be the number of solutions to the  congruence~\eqref{eq:cong llu} 
with a square-free $u$.  Using the standard inclusion-exclusion principle, we write 
$$
N_{a,p}^\sharp(L,h) = \sum_{d \le h^{1/2}} \mu(d) N_{ad^{-2},p}(L,h/d^2).
$$
We use Lemma~\ref{lem:congr-asymp} to estimate the contribution from $d \le D$ 
as 
\begin{align*} 
\sum_{d \le D} \mu(d) N_{ad^{-2},p}(L,h/d^2) & =
\frac{K^2 h}{p}  \sum_{d \le D} \frac{\mu(d)}{d^2}  +  O\(D L^{3/2} p^{1/8+o(1)}\)\\
& = \frac{K^2 h}{p}  \sum_{d=1}^\infty \frac{\mu(d)}{d^2}  +  O\( \frac{K^2 h}{Dp} + D L^{3/2} p^{1/8+o(1)}\)\\
& =  \frac{K^2 h}{\zeta(2) p}   +  O\( \frac{K^2 h}{Dp} + D L^{3/2} p^{1/8+o(1)}\),
\end{align*}
where, as before $K$, is the cardinality of the set  of   primes $\ell \in [L,2L]$.

Next, we use Lemma~\ref{lem:congr-bound} to estimate the contribution from $d > D$ 
as 
\begin{equation}
\begin{split}
\label{eq:large d}
\sum_{D < d \le h^{1/2}}  N_{ad^{-2},p}(L,h/d^2) & =
\sum_{D < d \le h^{1/2}} \(\frac{L^2 h}{d^2p}  +1\)p^{o(1)} \\
& \le \(\frac{L^2 h}{Dp} + h^{1/2}\)p^{o(1)}.
 \end{split}
\end{equation}
 Therefore,  we see that 
$$
N_{a,p}^\sharp(L,h) =   \frac{ K^2 h}{\zeta(2) p}      +  O\( \(\frac{L^2 h}{Dp} + D L^{3/2} p^{1/8}+h^{1/2}\) p^{o(1)}\). 
$$
 In particular,  recalling the choice of parameters in~\eqref{eq:DhL}, we obtain
\begin{equation}
\label{eq:NLhK4}
N_{a,p}^\sharp(L,h)=   \frac{ K^2 h}{\zeta(2) p}    +  
 O\(p^{1/2+7\varepsilon/4+o(1)}\). 
\end{equation} 
Since $K= L^{1+o(1)} =  p^{1/4 + \varepsilon+o(1)}$, the main term in~\eqref{eq:NLhK4}
is of the form $L^{2+o(1)}hp^{-1}  = p^{1/2+2\varepsilon + o(1)}$  which dominates the 
error term. Hence we can simplify the equation~\eqref{eq:NLhK4} as\begin{equation}
\label{eq:NLh4}
N_{a,p}^\sharp(L,h) =  p^{1/2+2\varepsilon + o(1)}. 
\end{equation} 

 Now a product $\ell_1\ell_2 u$ contributing to  $N_{a,p}^\sharp(L,h)$ is not square-free
 if  
 $$
 \ell_1 = \ell_2\quad \text{or} \quad  \ell_1 \mid u \quad \text{or} \quad  \ell_2 \mid u.
$$

Since  each choice   $\ell_1 = \ell_2$ defines $u$ uniquely, 
we have at most $L$ non square-free solutions of this type.

Solutions with  $\ell_j \mid u$, $j =1, 2$ lead to a congruence of the type~\eqref{eq:cong ll2v}
with $h/L$ instead of $h$. Thus, by  Lemma~\ref{lem:congr-boundsquare} there are $Lp^{o(1)}$ such solutions. 

Since both these quantities are much smaller than $N_{a,p}^\sharp(L,h)$
given by~\eqref{eq:NLh4}, and $\varepsilon$ is arbitrary, the result follows.

\subsection{Proof of Theorem~\ref{thm:M1-AA}}

Let us fix some $\varepsilon > 0$ and instead of~\eqref{eq:DhL} we now set
\begin{equation}
\label{eq:DQL}
D = Q^{\varepsilon/2}, \qquad E = Q^{1/6+\varepsilon},\qquad h = Q, \qquad L = Q^{1/6+\varepsilon}. 
\end{equation}

We follow the same lines as in the proof of Theorem \ref{thm:M1} using Lemma~\ref{lem:congr-asymp-AA} instead of Lemma~\ref{lem:congr-asymp}. To begin, we note that for $L \le p^{1/5}$, the bound of Lemma~\ref{lem:congr-asymp-AA} with $k=5$ 
takes the form
$$
\( L^{14/10} p^{1/10} + L^{19/10}\) p^{o(1)} =  L^{14/10} p^{1/10+o(1)}.
$$

Hence for any prime $p \in [Q,2Q]$ to which this bound applies, 
as in the proof of Theorem~\ref{thm:M1}, we estimate the contribution from $d \le D$ 
as 
$$
\sum_{d \le D}  \mu(d)  N_{ad^{-2},p}(L,h/d^2)  =  \frac{K^2 h}{\zeta(2) p}   +  O\( \frac{K^2 h}{Dp} + D  L^{14/10} p^{1/10 +o(1)} \), 
$$
where, as before $K= L^{1+o(1)}$, is the cardinality of the set  of   primes $\ell \in [L,2L]$. 

Next, using Lemma~\ref{lem:congr-bound} we estimate the contribution from $E \ge d > D$ similarly to~\eqref{eq:large d}
as 
\begin{align*} 
\sum_{D < d \le E}  N_{ad^{-2},p}(L,h/d^2) & =
\sum_{D < d \le E} \(\frac{L^2 h}{d^2p}  +1\)p^{o(1)} \\
& \le \(\frac{L^2 h}{Dp} +  E\)p^{o(1)}.
\end{align*}

Finally for large divisors,  precisely  $E < d \le h^{1/2}$ we cover the range of summation 
over $d$ by $O(\log p)$ dyadic intervals of the form $[F,2F]$ 
where $E \le F \le h^{1/2}$.  
Clearly,  recalling~\eqref{eq:DQL}, we verify that   
$$F \le   h^{1/2} \le  p \mand  L^2h/F^2 \le L^2h/E^2  \leq p.
$$
Hence, we can estimate the contribution from each 
of the intervals  by Lemma~\ref{lem:congr-bound-aver} as
\begin{align*} 
R_{a,p}(F,L,h/F^2) & \le    \max\left\{F(L^2h/F^2)^{1/4}p^{-1/4}, F^{1/2}(L^2h/F^2)^{1/4}\right\}  p^{o(1)}\\
& =       \max\left\{F^{1/2} (L^2h)^{1/4}p^{-1/4}, (L^2h)^{1/4}\right\}  p^{o(1)} \\
&=  (L^2h)^{1/4}p^{o(1)}
\end{align*} 
since $F \le h^{1/2} \le p^{1/2}$.

  Therefore,  we see that 
\begin{equation*} 
N_{a,p}^\sharp(L,h)
 =   \frac{ K^2 h}{\zeta(2) p}      +  O\( \(\frac{L^2 h}{Dp} + D L^{14/10} p^{1/10} +E + \(L^{2}h\)^{1/4} \) p^{o(1)}\). 
\end{equation*} 
 In particular,  recalling the choice of parameters in~\eqref{eq:DQL}, we obtain
 $L^2h = Q^{4/3+2\varepsilon}$ and thus 
 $$
\frac{L^2 h}{Dp} =Q^{1/3+(3/2)\varepsilon} , \quad  D L^{14/10} p^{1/10} \le Q^{1/3+19\varepsilon/10}, 
$$
while 
$$
 \(L^{2}h\)^{1/4} =  Q^{1/3+\varepsilon/2}.
 $$
 
 In particular
\begin{equation}
\label{eq:NLhK}
N_{a,p}^\sharp(L,h)=   \frac{ K^2 h}{\zeta(2) p}    +  
 O\(  Q^{1/3+19\varepsilon/10}\). 
\end{equation} 
Since $K= L^{1+o(1)} =  p^{1/6 + \varepsilon+o(1)}$, the main term in~\eqref{eq:NLhK}
is of the form $L^{2+o(1)}hp^{-1}  = Q^{1/3+2\varepsilon + o(1)}$  which dominates the 
error term. Hence, we can simplify the equation~\eqref{eq:NLhK} as
\begin{equation}
\label{eq:NLh}
N_{a,p}^\sharp(L,h) =   Q^{1/3+2\varepsilon + o(1)}. 
\end{equation} 

 Now a product $\ell_1\ell_2 u$ contributing to  $N_{a,p}^\sharp(L,h)$ is not square-free
 if  
 $$
 \ell_1 = \ell_2\quad \text{or} \quad  \ell_1 \mid u \quad \text{or} \quad  \ell_2 \mid u.
$$

Since  each choice   $\ell_1 = \ell_2$ defines $u$ uniquely, 
we have at most $L$ non square-free solutions of this type.

Solutions with  $\ell_j \mid u$, $j =1, 2$ lead to a congruence of the type~\eqref{eq:cong ll2v}
with $h/L$ instead of $h$. Thus, by  Lemma~\ref{lem:congr-boundsquare} there are $Lp^{o(1)}$ such solutions. 

Since both these quantities are much smaller than $N_{a,p}^\sharp(L,h)$
given by~\eqref{eq:NLh}, and $\varepsilon$ is arbitrary, the result follows.

\section{Comments}
\label{sec:comm}

Clearly any improvement of Lemma~\ref{lem:CharSquarfree}
immediately leads to an improvement of  Theorem~\ref{thm:Malpha}. In particular, 
it is mentioned in~\cite{Mun} that under the Generalised Riemann
Hypothesis (GRH) the bound 
$$
\max_{\chi\in\cX_p^*}
\left|S^\sharp_\chi(t)\right|\le t^{1/2}p^{ o(1)}, 
$$
holds for any  integer $t <p$. 
Combined with the  argument used in the proof of Theorem~\ref{thm:Malpha-p-AA}, 
it leads to the bound $M_{\alpha}^*(p) \le p^{2+o(1)}$ for any fixed $\alpha > 0$. 
Thus, Theorem~\ref{thm:Malpha-p-AA} shows unconditionally, that this 
holds for an overwhelming majority of primes $p$.

We recall that nontrivial upper bounds on $T_{a,p}(U,V)$ are also known for all $p$, albeit weaker than that of Lemma~\ref{lem:TIaa}.  For example,  Nunes~\cite[Equation~(3.13)]{Nun} gives the bound
$$
T_{a,p}(U,V) \le \min\left\{U^{2/3} V^{1/4}, U^{1/4} V^{2/3}\right\} p^{o(1)}.
$$
which holds for any integer $a$ with $\gcd(a,p)=1$ and reals $U$ and $V$ 
with  $1\le U,V \le p^{3/4}$. This, however, is not enough to improve 
Theorem~\ref{thm:M1}, where the bottleneck comes from bounds of 
double Kloosterman sums in Lemma~\ref{lem:BilinSums}. On the other hand,
a version of the argument used in the proof of Lemma~\ref{lem:TIaa} has been 
used in~\cite{LSZ2} to improve the result of Nunes~\cite{Nun} on squarefree numbers in   arithmetic progressions modulo $q$ on average over moduli $q$. 

Supported by the results of Theorem~\ref{thm:M1} and~\ref{thm:M1-AA}, we believe in the following

\begin{conj} As $p\to \infty$, we have
$$
M(p) =  p^{1+ o(1)}.
$$
\end{conj}

On the other hand, Andrew Booker has given a construction which shows that there is 
an absolute constant $c$ such that 
\begin{equation}
\label{eq:Low bound}
M(p) \ge c p \frac{\log p}{\log \log p}.
\end{equation} 
Indeed, for an integer  parameter $K$ we choose a prime $p$ such that 
$$
 kp+4  \equiv 0 \pmod {p_{k+1}^2}, \qquad   k=1,...,K,
$$
where $p_k$ denotes the $k$th prime. Clearly,  the smallest positive square-free $s \equiv 4 \pmod p$
satisfies $s > Kp$. By the Linnik theorem~\cite[Theorem~18.7]{IwKow} we can take 
$$
p  = \(\prod_{k=2}^K p_{k+1}^2\)^{O(1)} = \exp \(O(K\log K)\)
$$
which implies~\eqref{eq:Low bound}.

\section*{Acknowledgement}

The authors are grateful to  Andrew  Booker  and Carl Pomerance for  
discussions and encouragement.  In particular, Andrew  Booker provided an argument 
leading to the lower bound~\eqref{eq:Low bound}. 

This work was also partially supported (for M.M.) by the Austrian Science Fund (FWF), START-project Y-901 ``Probabilistic methods in analysis and number theory'' led by Christoph Aistleitner and (for I.S.) by the Australian Research Council Grant~DP170100786.

\end{document}